\documentclass[leqno]{article}
\usepackage{geometry}
\usepackage{graphicx}	
\usepackage[cp1251]{inputenc}
\usepackage[english]{babel}
\usepackage{mathtools}
\usepackage{amsfonts,amssymb,mathrsfs,amscd,amsmath,amsthm}
\usepackage{verbatim}

\usepackage{url}

\def\ig#1#2#3#4{\begin{figure}[!ht]\begin{center}%
\includegraphics[height=#2\textheight]{#1.eps}\caption{#4}\label{#3}%
\end{center}\end{figure}}

\def\thtext#1{
  \catcode`@=11
  \gdef\@thmcountersep{. #1}
  \catcode`@=12
}

\def\threst{
  \catcode`@=11
  \gdef\@thmcountersep{.}
  \catcode`@=12
}

\theoremstyle{plain}
\newtheorem{thm}{Theorem}[section]
\newtheorem{prop}[thm]{Proposition}
\newtheorem{cor}[thm]{Corollary}

\newtheorem{lem}[thm]{Lemma}

\theoremstyle{definition}

\newtheorem{examp}[thm]{Example}
\newtheorem{prb}[thm]{Problem}
\newtheorem{dfn}[thm]{Definition}
\newtheorem{rk}[thm]{Remark}

 \catcode`@=11
 \def\.{.\spacefactor\@m}
 \catcode`@=12

\def\N{{\mathbb N}}
\def\Q{{\mathbb Q}}
\def\R{\mathbb R}

\def\Z{{\mathbb Z}}

\def\a{\alpha}

\def\e{\varepsilon}
\def\dl{\delta}
\def\D{\Delta}
\def\g{\gamma}

\def\l{\lambda}
\def\L{\Lambda}
\def\r{\rho}

\def\t{\tau}
\def\v{\varphi}

\def\0{\emptyset}
\def\:{\colon}
\def\<{\langle}
\def\>{\rangle}

\def\d{\partial}

\def\rom#1{\emph{#1}}
\def\({\rom(}
\def\){\rom)}
\def\sm{\setminus}
\def\ss{\subset}

\def\sp{\supset}

\def\toGH{\xrightarrow{\GH}}

\def\x{\times}

\def\bX{{\bar X}}

\def\CC{\operatorname{CC}}

\def\diam{\operatorname{diam}}
\def\dil{\operatorname{dil}}
\def\dis{\operatorname{dis}}

\def\GH{\operatorname{\mathcal{G\!H}}}

\def\Iso{\operatorname{Iso}}

\def\VGH{\operatorname{\mathcal{VG\!H}}}

\def\cA{\mathcal{A}}

\def\cB{\mathcal{B}}
\def\cC{\mathcal{C}}

\def\cH{\mathcal{H}}

\def\cR{\mathcal{R}}

\begin{document}
\title{Gromov--Hausdorff class: its completeness and cloud geometry}
\author{S.~A.~Bogaty, A.~A.~Tuzhilin}
\maketitle

\begin{abstract}
The paper is devoted to the study of the Gromov--Hausdorff proper class, consisting of all metric spaces considered up to isometry. In this class, a generalized Gromov--Hausdorff pseudometric is introduced and the geometry of the resulting space is investigated. The first main result is a proof of the completeness of the space, i.e., that all fundamental sequences converge in it. Then we partition the space into maximal proper subclasses consisting of spaces at a finite distance from each other. We call such subclasses clouds. A multiplicative similarity group operates on clouds, multiplying all the distances of each metric space by some positive number. We present examples of similarity mappings transferring some clouds into another ones. We also show that if a cloud contains a space that remains at zero distance from itself under action of all similarities, then such a cloud contracted to this space. In the final part, we investigate subsets of the real line with respect to their behavior under various similarities.
\end{abstract}

\section*{Introduction}
\markright{\thesection.~Introduction}
This work is devoted to the study of the geometry of the Gromov--Hausdorff proper class~\cite{Mendelson}, which consists of all non-empty metric spaces considered up to isometry. This class is endowed with the famous Gromov--Hausdorff distance~\cite{Edwards, Gromov1981, Gromov1999}, which turns the class into a generalized pseudometric space (``generalize'' means that infinite distances may occur, and the prefix ``pseudo-'' that the distances between non-isometric spaces may vanish). Traditionally, this distance is studied on the Gromov--Hausdorff space, in which all metric spaces are compact~\cite{BurBurIva}. For non-compact spaces, a modified distance is usually introduced --- the pointed Gromov--Hausdorff distance~\cite{Herron}. In this case, one chooses some points in metric spaces and makes sure that when the spaces are compared with each other to define the distance, the points turn out to be ``close'' to each other. However, it is more traditional for non-compact spaces to consider the corresponding convergence without pre-setting the distance function~\cite{BurBurIva}.

The decision to ``break the tradition'' was implemented in~\cite{Borzov}, where the unmodified Gromov--Hausdorff distance is studied without restriction on compactness. In the present work, we continue this study. It was noted in~\cite{Borzov} that although a pseudometric can be defined on a proper class in the standard way, nevertheless, it is not possible to introduce the topology directly. Indeed, recall that in von Neumann--Bernays--G\"odel set theory, all objects are called classes that belong to one of two types: sets and proper classes. A class is called a set if it is a member of some other class. If the class is not a member of any other class, then it is called proper. Thus, it is not be possible to determine a topology on a proper class, because otherwise this class itself is an element of the topology and, therefore, must be a set. Nevertheless, for classes in which the cardinality filtering consists of sets, a ``topology'' can be introduced by assigning a certain topology to each element of the filtering, and requiring a natural matching of these topologies. This is how the concept of a topological class was introduced in~\cite{Borzov}. Also in~\cite{Borzov} it was shown that the Gromov--Hausdorff class is filtered by sets and, therefore, the Gromov--Hausdorff distance allows us to turn this class into a topological one.

The presence of such a ``topology'' allows one to define a continuous mapping from a topological space to a topological class\footnote{The presence of a pseudometric makes it possible to determine the continuity of a mapping from a topological space in the standard way, using spherical neighborhoods generated by the pseudometric. However, the neighborhoods can be proper subclasses, so it won't be possible to build a class from them. The topological class defined in~\cite{Borzov} allows working in more familiar terms.}, and with them continuous curves, their lengths, as well as the notion of intrinsic and strictly intrinsic generalized pseudometric. In~\cite {Borzov} it is proved that the Gromov--Hausdorff distance is intrinsic, that is, for metric spaces at a finite distance, it is equal to the infimum of the lengths of the curves connecting these spaces. The question of whether this distance is strictly intrinsic remains open. Recall that in the Gromov--Hausdorff space, the distance is a metric (it is always finite and positively definite) and, as shown in~\cite{IvaNikolaevaTuz}, this metric is strictly intrinsic. It is also well known~\cite{BurBurIva} that the Gromov--Hausdorff space is complete and separable.

What can we say about the geometry and topology of the Gromov--Hausdorff class? In this paper we show that the Gromov--Hausdorff class is complete. Then we split the Gromov--Hausdorff class into ``clouds'' --- maximal proper subclasses, each of which consists of spaces at a finite distance from each other. A cloud containing a one-point metric space includes exactly all bounded metric spaces. All other clouds are made up of unbounded spaces.

To formulate the following result, we introduce some notation. If $X$ is a metric space, then for a real $\l>0$, by $\l X$ we denote the metric space obtained from $X$ by multiplying all distances by $\l$. It is well known that for bounded metric spaces $X$ and $Y$, the Gromov--Hausdorff distance between $\l X$ and $\l Y$ equals to the distance between $X$ and $Y$ multiplying by $\l$. In addition, the distance from the one-point space $\D_1$ to a space $X$ equals to half the diameter of $X$. Thus, for $\l\to 0+$, all bounded spaces are contracted to $\D_1$. It is also clear that multiplication by $\l$ transfers the subclass of bounded spaces into itself. What about this operation for other clouds?

It turns out that everything is much more interesting here. We show by constructing a concrete example that for elements of some clouds, multiplication by $\l$ can lead to a jump to other clouds. In other words, for a positive $\l$, the distance between a metric space $X$ and $\l X$ may be infinite. Moreover, if for some $\l$ the space $\l X$ jumped out of a cloud containing $X$, then, when multiplied by another $\l$, the space $\l X$ can return to the original cloud. There are also such $X$ that never return: when multiplied by any positive $\l\ne 1$, the space $\l X$ is at the infinite distance from the space $X$. Naturally, the problem arises of finding out when the space jumps out, when it stays, when it returns, and in general how the corresponding families of $\l$ are arranged. In the final part of the paper, we will present a number of results on this topic using the example of metric spaces that are subsets of the real line.

Another interesting objects of the research are spaces that remain isometric to themselves when multiplied by any positive $\l$. We called such spaces $\dil$-invariant. Example of $\dil$-invariant spaces are a normed space $\R^n$, a positive orthant, a bouquet of $\dil$-invariant spaces, etc. An interesting problem is to describe both $\dil$-invariant spaces and the spaces at zero distance from them.

It turns out that if a cloud contains a $\dil$-invariant space, then the entire cloud is contracted to this space (in the case of $\R^n$ this was discovered by Ilya Belalov). For a cloud containing spaces jumping out of the cloud when multiplied by some $\l>0$, we show that all other spaces from this cloud behave in the same way. More precisely, the set of all $\l>0$, multiplication by which preserves a space in its cloud, is the same for all spaces of this cloud.

Thus, in this paper we prove a fundamental fact about the completeness of the Gromov--Hausdorff class and pose the problem of studying clouds for their ``stability'' w.r.t. multiplication by positive $\l$. Many interesting questions remain outside the scope of this work, for example, what could be the set of those $\l$ that preserve the cloud? It is easy to verify that these sets form multiplicative groups. Can any multiplicative subgroup be so implemented? Another question: which clouds are stable and which ones are not? For example, how can one describe the sequences of real numbers that remain in their clouds? It is clear that all arithmetic progressions are of this kind. What about increasing geometric progressions? We present a number of results describing the non-trivial behavior of these progressions.

The authors are grateful to the team of the seminar by A.O. Ivanov and A.A. Tuzhilin ``Theory of extremal networks'', and especially to its co-director Alexander Ivanov, and also to participant Ilya Belalov, for fruitful discussions. The study was carried out at Lomonosov Moscow State University, while A.A.Tuzhilin was supported by Russian Science Foundation, project 21-11-00355.

\section{Basic definitions and preliminary results}
\markright{\thesection.~Basic definitions and preliminary results}
As noted in the introduction, we will study the Gromov--Hausdorff distance on the proper class of all metric spaces. Let us recall the corresponding definitions.

In von Neumann-Bernays-G\"odel (NGB) set theory, all objects are called \emph{classes}. Classes are of two types: \emph{sets}, which are defined as classes that are elements of other classes, and \emph{proper classes}, which are not elemets of any other classes. Many standard operations are defined on the classes, for example, Cartesian product, mappings, etc. In particular, a distance function can be specified on each class. We will use the following terminology:
\begin{itemize}
\item a \emph{distance function\/} or, in short, a \emph{distance\/} on a class $\cA$ is an arbitrary mapping $\r\:\cA\x\cA\to[0,\infty]$, for which always $\r(x,x)=0$ and $\r(x,y)=\r(y,x)$ (symmetry) hold;
\item if the distance satisfies the triangle inequality, i.e., if $\r(x,z)\le\r(x,y)+\r(y,z)$ is always satisfied, then $\r$ is called \emph{generalized pseudometric\/} (the word "generalized'' corresponds to the possibility of taking the value $\infty$);
\item if the positive definiteness condition is additionally satisfied for a generalized pseudometric, i.e., if $\r(x,y)=0$ always implies $x=y$, then we call such $\r$ \emph{generalized metric\/};
\item finally, if $\r$ does not take the value $\infty$, then in the above definitions we will omit the word `generalized", and sometimes, to emphasize the absence of $\infty$, we call this distance \emph{finite}.
\end{itemize}

As is customary in metric geometry, instead of $\r(x,y)$ we write $|xy|$ as a rule.

As we already noted in the introduction, it is impossible to define a topology on a proper class, since otherwise the proper class will be an element of this topology and, therefore, will be a set. Nevertheless, we have proposed a modified construction, which is also called topology, but already defined on proper classes.

Let $\cA$ be an arbitrary class. Since all its elements are sets, the cardinality is defined for each of them. For each cardinal number $n$, consider a subclass $\cA_n$ consisting of all elements whose cardinalities do not exceed $n$. Then there is a natural filtering: $\cA_m\ss\cA_n$ for $m\le n$. This filtering will be called \emph{filtering by sets\/} if all $\cA_n$ are sets. For a class that is filtered by sets, we define \emph{topology\/} as a mapping $\t\:n\mapsto\t_n$ that maps each cardinal number to a topology on $\cA_n$ so that these topologies are \emph{consistent}, i.e., if $m<n$, then the topology $\t_m$ is induced from the topology $\t_n$. A class filtered by sets with the topology defined on it will be called a \emph{topological class}.

Note that if we take a set of cardinality $n$ as a class $\cA$ and define a topology $\t\:m\to\cA_m$ as described above, then the topology $\t_n$ on $\cA_n=\cA$ induces topologies on all $\cA_m$, $m<n$. On all $\cA_m$, $m>n$, coinciding in this case with $\cA$, the topology $\t_m$ will coincide with $\t_n$. Thus, in this case we actually have a standard topology on the set $\cA$ (all other topologies are obtained in the standard way from it), so our construction can be considered as a generalization of the concept of topology to proper classes.

Further, if a generalized pseudometric (or metric) is given on a class, then it turns this class into a topological one in a standard way. The corresponding topology will be called \emph{pseudometric\/} (\emph{metric\/}).

It follows from the NGB axioms that if $X$ is a set, $\cA$ is a class, and $f\:X\to\cA$ is a mapping, then the image $f(X)$ is a set that belongs to some $\cA_n$. Thus, if we take a topological space as $X$ and a topological class as $\cA$, then the \emph{continuity\/} of the mapping $f$ is naturally defined, namely, the mapping $f\:X\to\cA_n$ \emph{is continuous\/} if its restriction $X\mapsto\cA_n$ is continuous for some and, therefore, for any cardinal $n$ satisfying the condition $f(X)\ss\cA_n$.

The latter allows us to introduce in the standard way the definition of a continuous curve $\g$, and with it, in the case when the topology is generated by a generalized pseudometric, the concepts of the length $|\g|$ of a continuous curve $\g$, intrinsic pseudometric, as well as strictly intrinsic pseudometric and geodesic class. Namely, a generalized pseudometric on a set-filtered class $\cA$ will be called \emph{intrinsic\/} if for every $X,Y\in\cA$, $|XY|<\infty$, and for any $\e>0$ there exists a continuous curve $\g$ connecting $X$ and $Y$ such that $|\g|<|XY|+\e$.

Let us now recall the definitions we need from metric geometry, namely, the Hausdorff and Gromov--Hausdorff distances.

\subsection{Hausdorff distance}
Let $X$ be an arbitrary metric space, $x\in X$, $r>0$ and $s\ge 0$ be real numbers. By $U_r(x)$ and $B_s(x)$ we denote, respectively, the \emph{open\/} and \emph{closed balls\/} centered at the point $x$ and with the radii $r$ and $s$. If $A$ and $B$ are non-empty subsets of $X$, then we put $|XA|=|AX|=\inf\bigl\{|xa|:a\in A\bigr\}$ and $|AB|=|BA|=\inf\bigl\{|ab|:a\in A,\,b\in B\bigr\}$. Next, we define the \emph{open $r$-neighborhood of the set $A$} by setting $U_r(A)=\bigl\{x\in X:|xA|<r\bigr\}$. Finally, the \emph{Hausdorff distance\/} between $A$ and $B$ is the value
$$
d_H(A,B)=\inf\bigl\{r:A\ss U_r(B)\ \&\ U_r(A)\sp B\bigr\}.
$$
The Hausdorff distance is a generalized pseudometric: it can be infinite, as in the case of the straight line $\R$ and any of its points, and also equal to zero between different subsets, for example, between the segment $[0,1]$ and the interval $(0,1)$. Nevertheless, the following classical result holds.

\begin{thm}[\cite{BurBurIva}]
On the set $\cH(X)$ consisting of all non-empty bounded closed subsets of the metric space $X$, the Hausdorff distance is a metric. Moreover, $X$ and $\cH(X)$ simultaneously possess or not the following properties: completeness, total boundedness, compactness, bounded compactness.
\end{thm}

\subsection{Gromov--Hausdorff distance}
We denote by $\VGH$ the proper class consisting of all non-empty metric spaces. On this class, we define a distance function called the \emph{Gromov--Hausdorff distance\/}:
$$
d_{GH}(X,Y)=\inf\bigl\{d_H(X',Y'):X',Y'\ss Z\in\VGH,\,X'\approx X,\,Y'\approx Y\bigr\},
$$
where for the metric spaces $U$ and $V$ the expression $U\approx V$ means that these spaces are isometric.

The following theorem is well known.

\begin{thm}[\cite{BurBurIva}]
The Gromov--Hausdorff distance is a generalized pseudometric vanishing on each pair of isometric spaces.
\end{thm}

This theorem allows us to investigate the Gromov--Hausdorff distance on a ``less wild'' proper class $\GH$ consisting of representatives of isometry classes of all metric spaces, one from each isometry class. This is similar to considering all sets up to bijection, that is, to consider the proper class of all cardinal numbers.

\begin{thm}[\cite{Borzov}]\label{thm:classGH}
The proper class $\GH$ is filtered by sets, so that the Gromov--Hausdorff distance turns it into a topological class. Moreover, the Gromov--Hausdorff distance is an intrinsic generalized pseudometric on $\GH$.
\end{thm}

The topological class $\GH$ described in Theorem~\ref{thm:classGH} will be called the \emph{Gromov--Hausdorff class}. An important subclass in $\GH$ is the proper class which consists of all bounded metric spaces, i.e.,  spaces $X$ with finite diameter $\diam X$. This class will be denoted by $\cB$.

Let us describe a few simplest properties of the class $\GH$. By $\D_1\in\GH$ we denote the one-point metric space. Also, for a space $X\in\GH$ and a real number $\l>0$, we denote by $\l X$ the metric space that is obtained from $X$ by multiplying its all distances by $\l$. If $\l=0$ and $\diam X<\infty$, we put $\l X=\D_1$. The transformation $H_\l\:\GH\to\GH$, $H_\l\:X\mapsto\l X$ for $\l>0$, we call \emph{similarity with coefficient $\l$}. 

\begin{thm}[\cite{BurBurIva}]\label{thm:estim}
For any $X,Y\in\GH$,
\begin{enumerate}
\item\label{htm:estim:1} $2d_{GH}(\D_1,X)=\dim X$\rom;
\item\label{tm:estima:2} $2d_{GH}(X,Y)\le\max\{\dim X,\dim Y\}$\rom;
\item\label{tm:estima:3} if the diameter of $X$ or $Y$ is finite, then
$\bigl|\dim X-\dim Y\big|\le2d_{GH}(X,Y)$.
\item\label{thm:estim:4} if the diameter of $X$ is finite, then for any $\l\ge0$ and $\mu\ge0$ we have $d_{GH}(\l X,\mu X)=\frac12|\l-\mu|\diam X$, whence it immediately follows that the curve $\g(t):=t\,X$ is shortest between any of its points, and the length of such a segment of the curve is equal to the distance between its ends\rom;
\item\label{thm:estim:5} for any $\l\ge0$, we have $d_{GH}(\l X,\l Y)=\l\,d_{GH}(X,Y)$. Note that the only space $X\in\cB$ for which $\l X$ is isometric to $X$ for all $\l>0$, is the one-point space $\D_1$. Thus, the similarity transformation $H_\l$ is a homothety of the space $\cB$ centered at the one-point metric space.
\end{enumerate}
\end{thm}

Figure~\ref{fig:gh-space-eng} shows the scheme of the Gromov--Hausdorff class.

\ig{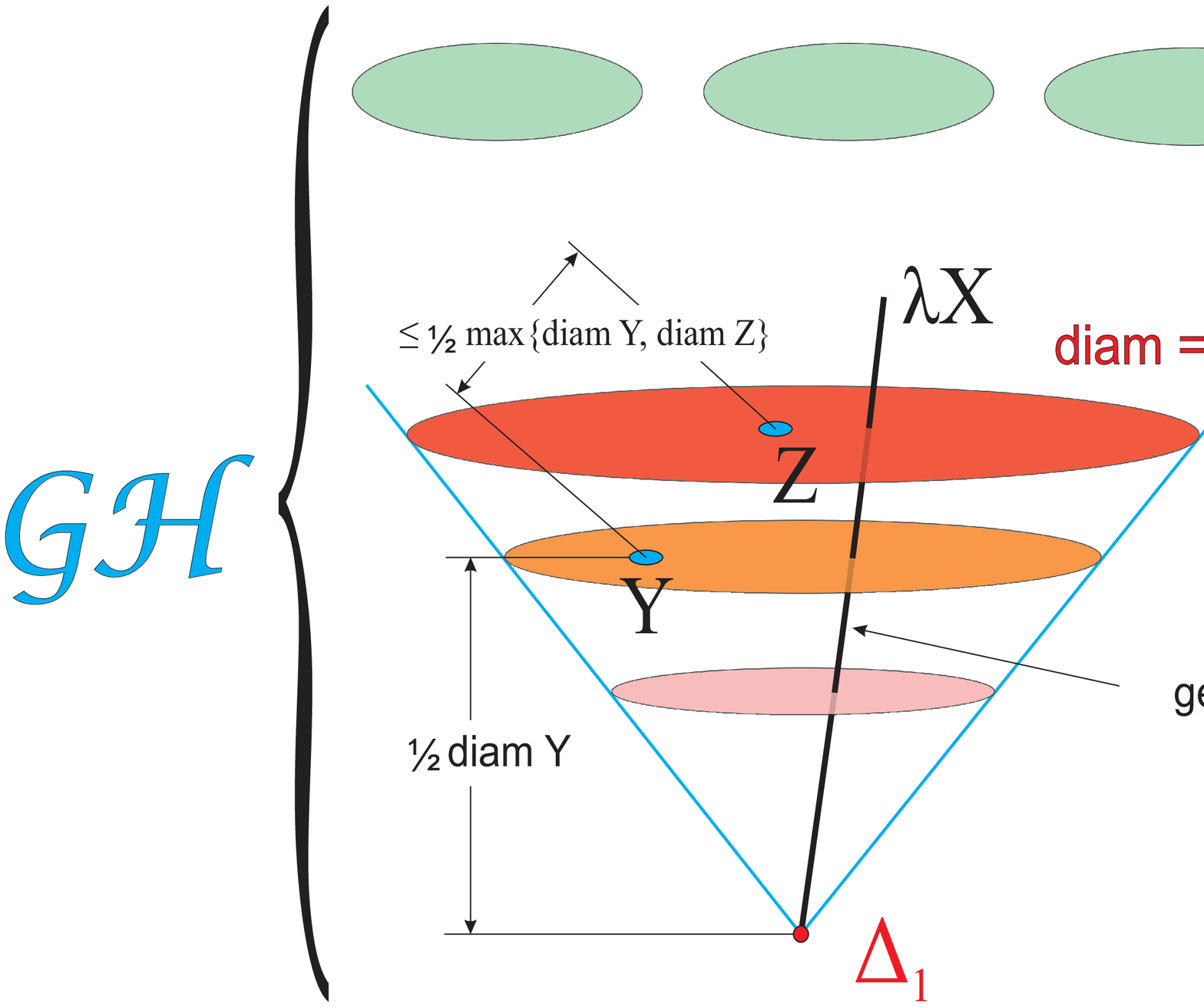}{0.36}{fig:gh-space-eng}{Gromov--Hausdorff class, general view.}

Consider the relation $\sim_1$ on $\GH$: we say that $X,Y\in\GH$ are in relation $\sim_1$ if and only if $d_{GH}(X,Y)=0$. It is easy to see that $\sim_1$ is an equivalence. The equivalence classes of this relation will be called \emph{clouds}. It is clear that the Gromov--Hausdorff distance between points of the same cloud is finite, and between points of different clouds is infinite. Thus, the restriction of the Gromov--Hausdorff distance to the cloud is a (finite) pseudometric.

\section{Completeness of the Gromov--Hausdorff class}
\markright{\thesection.~Completeness of the Gromov--Hausdorff class}
The purpose of this subsection is to prove a fundamental property of the Gromov--Hausdorff class.

\begin{thm}\label{thm:GHcomplete}
The Gromov--Hausdorff class is complete. In particular, all clouds are complete.
\end{thm}

\begin{proof}
Let $X_1,X_2,\ldots$ be a fundamental sequence. Without loss of generality, we will assume that $2d_{GH}(X_n,X_{n+1})\le1/2^n$. For each $n$, choose a correspondence $R_n\in\cR(X_n,X_{n=1})$ such that $\dis R_n<1/2^n$. We put $\bX=\sqcup_{n=1}^\infty$, and for each $x_1\in X_1$, we call the sequence $x_1,x_2\in\bX$ a \emph{thread starting at $x_1$} if for any $n\ge1$ it holds $x_{n+1}\in R_n(x_n)$. We denote the set of all such threads by $N(\bX)$. For threads $\nu=(x_1,x_2,\ldots)$ and $\nu'=(x'_1,x'_2,\ldots)$, consider the sequence of real numbers $|x_1x'_1|,|x_2x'_2|,\ldots$. Since $\dis R_n<1/2^n$ then $\bigl||x_nx'_n|-|x_{n+1}x'_{n+1}|\bigr|<1/2^n$, whence for any $m\ge1$ we have
\begin{multline*}
\bigl||x_nx'_n|-|x_{n+m}x'_{n+m}|\bigr|=
\bigl||x_nx'_n|-|x_{n+1}x'_{n+1}|+|x_{n+1}x'_{n+1}|-\cdots-|x_{n+m}x'_{n+m}|\bigr|\\
\le\sum_{k=1}^m\bigl||x_{n+k-1}x'_{n+k-1}|-|x_{n+k}x'_{n+k}|\bigr|<\sum_{k=1}^m\frac1{2^{n+k-1}}<\frac1{2^{n-1}},
\end{multline*}
therefore, the sequence of numbers $|x_1x'_1|,|x_2x'_2|, \ldots$ is fundamental and, therefore, there is a limit which we denote by $|\nu\nu'|$. It is clear that these limits define a distance function on $N(\bX) $. Since for any three threads $\nu=(x_1,x_2,\ldots)$, $\nu'=(x'_1,x'_2,\ldots)$, $\nu''=(x''_ 1,x''_2,\ldots)$, and for each $n$, the triangle inequalities hold for the points $x_n$, $x'_n$, $x''_n$, the same is true for the distances between $\nu$, $\nu'$, and $\nu''$. Thus, we have defined a pseudometric on $N(\bX)$. Having factorized the resulting pseudometric space with respect to the equivalence relation generated by zero distances, we obtain a metric space which we denote by $X$. For a thread $\nu\in N(\bX)$, the class of this equivalence containing $\nu$ we denote by $[\nu]$.

For each $n$, consider the relation $R'_n\ss X\x X_n$ defined as follows: for each $x$, consider all threads $\nu=(x_1,x_2,\ldots,x_n,\ldots)\in x$ and put all $(x,x_n)$ into $R'_n$. Since each $x_n$ belongs to some thread, then $R'\in\cR(X,X_n)$. Let us estimate the distortion of the correspondence $R '$. Choose arbitrary $x,x'\in X$, threads $\nu=(x_1,x_2,\ldots,x_n,\ldots)\in x$ and $\nu'=(x'_1,x'_2,\ldots,x'_n,\ldots)\in x'$, then $|\nu\nu'|=\lim_{k\to\infty}|x_kx'_k|$, and as shown above, $\bigl||x_nx'_n|-|x_{n+m}x'_{n+m}|\bigr|<1/2^{n-1}$, whence $\bigl||x_nx'_n|-|\nu\nu'|\bigr|\le1/2^{n-1}$. Thus, $\dis R'\le1/2^{n-1}$, therefore, $X_n\toGH X$.
\end{proof}

\section{Contracted and rain clouds}
\markright{\thesection.~Contracted and rain clouds}
Starting from this section, we will study how the similarity transformation changes the Gromov--Hausdorff distance, in particular, whether the metric space remains in the cloud, or jumps to another one.

\begin{cor}\label{cor:CloudContract}
Suppose that in a cloud $\cC$ there exists $X\in\cC$ such that for some $0<\l<1$, the condition $\l X\in\cC$ is fulfilled. Then the mapping $H_\l\:\GH\to\GH$, $H_\l\:Y\mapsto\l Y$, transforms the cloud $\cC$ into itself, and its restriction to $\cC$ is a contracting mapping. Thus, for any $X,Y\in\CC$, there are limits $X'=\lim_{n\to\infty}H_\l^n(X)$, $Y'=\lim_{n\to\infty}H_\l^n(Y)$, with $d_{GH}(X',Y')=0$.
\end{cor}

\begin{proof}
For any $Y\in\cC$, we have $d_{GH}(\l X,\l Y)=\l d_{GH}(X,Y)<\infty$, therefore $\l Y\in\cC$, and the mapping $H_\l$ is contracting. Since the class $\GH$ is complete by Theorem~\ref{thm:GHcomplete}, all sequences $H_\l^n(Z)$ converges as $n\to\infty$, and all their limits are at zero distance from each other (the uniqueness result on the limit of a contracting mapping in the case of pseudometric spaces).
\end{proof}

For $X$ from a cloud $\cC$, we denote by $\L_X$ the set of all $\l>0$ for which $\l X\in\cC$. The following result immediately follows from the equality $d_{GH}(\l X,\l Y)=\l d_{GH}(X,Y)$.

\begin{prop}\label{prop:AllGoodL}
The set $\L_X$ does not depend on the choice of $X\in\cC$.
\end{prop}

Proposition~\ref{prop:AllGoodL} allows us to give the following definition: for a cloud $\cC$, we put $\L_\cC=\L_X$ for arbitrary $X\in\cC$ and call it the \emph{cloud $\cC$ stabilizer}.

\begin{rk}
It is easy to see that the stabilizer $\L_\cC$ of the cloud $\cC$ is a subgroup in the multiplicative group of positive real numbers. In addition, Example~\ref{examp:couldRainy} (see below) shows that $\L_\cC$ does not preserve the additive structure: the sum of elements from $\L_\cC$ may belong not to $\L_\cC$.
\end{rk}

\begin{examp}\label{examp:Belalov}
Let $\cC$ be a cloud containing $\R^n$. Then for any $\l>0$, the spaces $\l\R^n$ and $\R^n$ are isometric. This implies that for any $X\in\cC$ and every $\l>0$, we have $d_{GH}(\l X,\R^n)=\l\,d_{GH}(X,\R^n)$, so $\l X\toGH\R^n$ as $\l\to0+$. In particular, $\L_\cC=(0,\infty)$.

This example was first considered by Ilya Belalov.
\end{examp}

\begin{dfn}
A metric space $X$ is called \emph{$\dil$-invariant\/} if, for all $\l>0$, the spaces $X$ and $\l X$ are isometric. In other words, $\dil$-invariance means that all similarities of a metric space are contained in the isometry group $\Iso(X)$ of this space.
\end{dfn}

\begin{prb}
Describe all $\dil$-invariant metric spaces.
\end{prb}

\begin{dfn}
If in a cloud $\cC$ there is a $Z$ such that for all $X\in\cC$ it holds $\l X\toGH Z$ as $\l\to0+$, then such a cloud is called \emph{contractible}.
\end{dfn}

Similarly with Example~\ref{examp:Belalov}, the following result can be proved.

\begin{cor}
Each cloud $\cC$ containing a $\dil$-invariant space is contractible. In particular, it holds $\L_\cC=(0,\infty)$.
\end{cor}

The definition of a $\dil$-invariant space can be generalized.

\begin{dfn}
We say that a metric space $X$ is a \emph{generalized $\dil$-invariant\/} if for any $\l>0$, we have $\d_{GH}(X,\l X)=0$.
\end{dfn}

\begin{cor}
Each cloud $\cC$ containing a generalized $\dil$-invariant space is contractible. In particular, it satisfies $\L_\cC=(0,\infty)$.
\end{cor}

\begin{prb}
Is it true that each cloud $\cC$ with $\L_\cC=(0,\infty)$ is contractible?
\end{prb}

\begin{dfn}
For $A\ss(0,\infty)$, a metric space $X$ is called \emph{$A$-invariant\/} if for all $\l\in A$ it holds $d_{GH}(X,\l X)=0$.
\end{dfn}

\begin{cor}
Let $\l\in(0,1)$ and $A=\{\l^r:r\ge0,\,r\in\Q\}$. Suppose that there exist a cloud $\cC$ and a metric space $X\in\cC$ such that for all $\l^r\in A$ the spaces $\l^rX$ belong to $\cC$. Let $Z$ be the limiting space for the fundamental sequence $X_n=H_\l^n(X)$. Then for all $\l^r\in A$, the space $Z$ is also a limiting one for the fundamental sequence $Y_n=H_{\l^r}^n(X)=H_{\l^{rn}}(X)$. Moreover, $Z$ is $A$-invariant.
\end{cor}

\begin{proof}
The sequence $\l^{rn}$ contains a subsequence $\l^{n_i}$, $n_i\in\N$. Indeed, if $r=p/q$, $p,q\in\Z$, $p\ge0$, $q>0$, then for $n=qi$ we get $\l^{rn}=\l_{pi}$. It remains to set $n_i=pi$. Thus, $d_{GH}(Y_{n_i},Z)\to0$ as $i\to\infty$, whence the fundamentality of the sequence $Y_n$ implies $Y_n\toGH Z$.

On the other hand, for each $r\in A$, the sequence $\l^{rn}$ contains a subsequence $\l^{n_i}\l^r$, $n_i\in\N$. Indeed, putting $r=p/q$ again and choosing $n$ of the form $1+qi$, we get $\l^{rn}=\l^{r+pi}=\l^{pi}\l^r$. It remains to put $n_i=pi$. Thus, $Z$ is the limit of the sequence $H_\l^{n_i}(\l^rX)$. However, the $\l^rZ$ is the limit of this sequence as well. From the ``uniqueness of the limit'' it follows that $d_{GH}(Z,\l^rZ)=0$. The proof is complete.
\end{proof}

\begin{dfn}
If $\L_\cC\ne(0,\infty)$, then we call the cloud $\cC$ \emph{rain}, and each of its elements a \emph{drop}.
\end{dfn}

\section{Rain clouds}
\markright{\thesection.~Rain clouds}
Let us show that rain clouds exist by constructing a series of examples of monotonically increasing sequences of real numbers, which, with an induced metric from the real line, are drops in their corresponding clouds. We start with a few general results.

\begin{lem}\label{lem:unbounded-and-big-xy}
Let $X,Y\in\GH$ be unbounded spaces for which there exists $R\in\cR(X,Y)$ with $\dis R<\infty$.
Then for each $x_0\in X$, $y_0\in Y$, and $r>0$, there exists $(x,y)\in R$ such that $x\in X\sm U_r(x_0)$ and $y\in Y\sm U_r(y_0)$.
\end{lem}

\begin{proof}
Suppose the opposite, i.e., for some $x_0\in X$ and $y_0\in Y$, there exists $r>0$ such that each $(x,y)\in R$ satisfies $\min\bigl\{|xx_0|,|yy_0|\bigr\}<r$. Since the space $Y$ is unbounded, there exists a sequence $y_k\in Y$ such that $|y_0y_k|\to\infty$ as $k\to\infty$. Since $R$ is a correspondence, for every $k$ there exists $x_k$ such that $(x_k,y_k)\in R$. But then, for all sufficiently large $k$, we have $|x_0x_k|<r$. Therefore, choosing such a $k$, we get
$$
\dis R\ge\sup_l\bigl||x_kx_l|-|y_ky_l|\bigr|=\infty,
$$
a contradiction.
\end{proof}

\begin{thm}\label{thm:unbounded-spaces-suff-infinity}
Suppose that for unbounded spaces $X,Y\in\GH$, the following condition holds\/\rom:
there exist $x_0\in X$ and $y_0\in Y$ such that
$$
\D(r)=\inf\Bigl\{\bigl||xx'|-|yy'|\bigr|:x,x'\in X\sm U_r(x_0),\ y,y'\in Y\sm U_r(y_0), \ x\ne x', y\ne y'\Bigr\}\to\infty
$$
for $r\to\infty$. Then for each correspondence $R\in\cR$ we have $\dis R=\infty$, i.e., $d_{GH}(X,Y)=\infty$.
\end{thm}

\begin{proof}
Suppose the opposite, i.e., let $\dis R<\infty$. Lemma~\ref{lem:unbounded-and-big-xy} implies that for every $r>0$, there exists $(x,y)\in R$ with $|x_0x|\ge r$ and $|y_0y|\ge r$. Choose $r'>r$ such that $r'>\max\bigl\{|x_0x|,|y_0y|\bigr\}$, then the pair $(x',y')\in R$ corresponding $r'$ satisfies $x\ne x'$ and $y\ne y'$. But then,
$$
\dis R\ge\bigl||xx'|-|yy'|\bigr|\ge\D(r),
$$
and since $r$ can be chosen arbitrarily large, we get $\dis R=\infty$, a contradiction.
\end{proof}

\begin{examp}\label{examp:couldRainy}
Let $p$ be a prime number greater than $2$, and $X=\{x_1=p,\,x_2=p^2,\,x_3=p^3,\ldots\}\ss\R$. Put $Y=2X$ and implement it as a sequence $y_i=2x_i$. Let us show that $d_{GH}(X,Y)=\infty$ and thus $X$ is a drop in its cloud $\cC$. To do that, we use Theorem~\ref{thm:unbounded-spaces-suff-infinity} to show that $\D(r)\to\infty$ as $r\to\infty$.

Choose arbitrary $x,x'\in X$, $x\ne x'$, and $y,y'\in Y$, $y\ne y'$, in such a way that
$$
\min\bigl\{|x_1x|,\,|x_1x'|,\,|y_1y|,\,|y_1y'|\bigr\}\ge r.
$$
Then for some positive integers $m$, $l$, $n$, $k$ we get $|xx'|=p^m-p^l$, $|yy'|=2(p^n-p^k)$. Since $x\ne x'$ and $y\ne y'$, then $m>l$ and $n>k$.

Put $s=\min\{m,l,n,k\}$, then the condition $r\to\infty$ is equivalent to $s\to\infty$. Further, let
$$
\dl=p^m-p^l-2p^n+2p^k=p^sz.
$$
Let us show that $z\ne0$, whence the required will follow, since $s$ can be chosen arbitrarily large, and $|z|\ge1$, therefore, $|\dl|\to\infty$ as $s\to\infty$, and $\D(r)\to\infty$ as $r\to\infty$, respectively, due to the arbitrariness of the choice $x$, $x'$, $y$, $y'$.

Let $z=0$, then $\D=0$, i.e., $p^l(p^{m-l}-1)=2p^k(p^{n-k}-1)$. Since $m-l>0$ and $n-k>0$, we get $l=k$, whence the equality $2p^{n-k}-p^{m-l}=1$ follows. Since $n-k>0$ and $m-l>0$, the left side of the equality is a multiple of $p$, a contradiction.
\end{examp}

\begin{rk}
Note that if in the construction from Example~\ref{examp:couldRainy} we multiply the space $X$ by $p$, then the space will remain in the cloud, since $d_{GH}(X,pX)<\infty$. Thus, it is not true that the existence of a metric space ``jumping out'' of its cloud for some $\l$, it follows that for all $\l>0$, other than $1$, the space $\l X$ does not lie in the cloud as well.
\end{rk}

\section{An example of studying the structure of a cloud stabilizer}
\markright{\thesection.~An example of studying the structure of a cloud stabilizer}

Let $\v\:\N\to\R$ be such a monotonically increasing function of a positive integer argument that $\v(n)\ge n$. For a real number $q>1$, consider the set
\begin{equation}\label{eq:1}
X_\v=\{x_n=q^{\v(n)}:n\in\N\}\ss\R
\end{equation}
with the metric induced from $\R$. We denote the cloud that contains the space $X_\v$ by the same symbol.

In the case of the function $\v(n)=n$, we denote the space $X_\v$ by $X_q$. Recall that by $H_\l\:\GH\to\GH$, $\l>0$, we denoted the similarity transformation $H_\l\:X\mapsto\l X$. Sometimes we will refer to the space $X_q$ as the space $\{x_n=q^n:n\in\Z\}\ss\R$, which defines the same cloud but is invariant under homothety $H_q$.

\begin{cor}\label{cor:lambda-in-Q}
If $q\in \N$ and $\l\in\L_{X_\v}\cap\Q$, then for any $M$ there exist $n>k>M$ and $m>l>M$ such that
\begin{equation}\label{eq:2}
\l=\frac{q^{\v(n)}-q^{\v(k)}}{q^{\v(m)}-q^{\v(l)}}.
\end{equation}
\end{cor}

\begin{proof}
Let us show that otherwise $d_{GH}(X,\l X)=\infty$. To do this, we use Theorem~\ref{thm:unbounded-spaces-suff-infinity} to show that $\D(r)\to\infty$ as $r\to\infty$.

Let $\l=a_2/a_1$ for some coprime positive integers $a_1$ and $a_2$.
Choose arbitrary $x,x'\in X$, $x\ne x'$, and $y,y'\in Y$, $y\ne y'$ such that
$$
\min\bigl\{|x_1x|,\,|x_1x'|,\,|y_1y|,\,|y_1y'|\bigr\}\ge r.
$$
Then for some positive integers $n$, $k$, $m$, $l$ we have
$|xx'|=q^{\v(n)}-q^{\v(k)}$, $|yy'|=\l (q^{\v (m)}-q^{\v (l)})$.
Since $x\ne x'$ and $y\ne y'$, then $n>k$ and $m>l$.

Let $s=\min\{n,k,m,l\}$, then the condition $r\to\infty$ is equivalent to $s\to\infty$. Further, let
$$
\dl=q^{\v (n)}-q^{\v (k)}-\l (q^{\v (m)}-q^{\v (l)})=q^{\v (s)}z.
$$
Suppose that $z\ne 0$. Since $s>M$ can be chosen as large as desired, and the number $a_1z$ is an integer, then $|a_1z|\ge 1$. Then $|\dl|\to\infty$ as $s\to\infty$ (since the number $a_1$ is fixed) and, respectively, $\D(r)\to\infty$ as $r\to\infty$ due to the arbitrariness of the choice of $x$, $x'$, $y$, $y'$.

Therefore, $z=0$, i.e., the number $\l$ has the required form.
\end{proof}

\begin{thm}\label{thm:2}
For a rational number $\l>0$ and an integer $q\ge 2$, the similarity transformation $H_\l$ keeps the space $X_q=\{q^n:n\in\Z\}$ in the same cloud, i.e., $\l\in\L_{X_q}\cap\Q$, if and only if $\l=q^\a$ for some $\a\in\Z$.
\end{thm}

We split the proof into several elementary steps some of which are undoubtedly well known and presented here for completeness and consistency.

For a given positive integers $n$, $m$, and an integer $q\geq 2$, we describe all solutions of the equation
\begin{equation}\label{eq:4}
a_1(q^n-1)=a_2(q^m-1)
\end{equation}
in positive integer coprime numbers $a_1$, $a_2$.

\begin{prop}
For any integer number $q\ge2$, we have
\begin{equation}
\label{eq:5}
\gcd(q^n-1,q^m-1)=q^{\gcd(n,m)}-1.
\end{equation}
\end{prop}

\begin{proof}
We use induction on the number $\max(n,m)$. If $n=m$, then the Equality~(\ref{eq:5}) is obvious. If $n>m$, then the inductive transition is given by the chain of equalities
\begin{multline}
\label{eq:6}
\gcd(q^n-1,q^m-1)=\gcd(q^n-1-q^{n-m}(q^m-1),q^m-1)= \\
=\gcd(q^{n-m}-1,q^m-1)=q^{\gcd (n-m,m)}-1=q^{\gcd (n,m)}-1.
\end{multline}
\end{proof}

\begin{cor}\label{cor:4}
If the positive integers $n$ and $m$ are coprime, then for every integer $q\geq 2$, the numbers
\begin{equation}
\label{eq:7}
q^{n-1}+q^{n-2}+\ldots +q+1\text{ and }q^{m-1}+q^{m-2}+\ldots+q+1 \\
\end{equation}
are also coprime.
\end{cor}

\begin{thm}
For given positive integers $n$, $m$, and an integer $q\ge2$, positive coprime integers $a_1$, $a_2$
solves the Equation~$(\ref{eq:4})$ if and only if
\begin{equation}\label{eq:8}
\begin{cases}
a_1=q^{(r_m-1)d}+q^{(r_m-2)d}+\ldots +q^d+1, \\
a_2=q^{(r_n-1)d}+q^{(r_n-2)d}+\ldots +q^d+1,
\end{cases}
\end{equation}
where $d=\gcd(n,m)$, $r_n=n/d$, and $r_m=m/d$.
\end{thm}

\begin{proof}
Substitute $x=q^d$ into the well-known equality
$$
x^r-1=(x-1)(x^{r-1}+\ldots+x+1)
$$
with $r=n/d$ and $r=m/d$. Thus, according to Equation~(\ref{eq:5}), the Equation~(\ref{eq:4}) can be rewritten as
\begin{equation}
\label{eq:9}
a_1(q^{(r_n-1)d}+q^{(r_n-2)d}+\ldots +q^{d}+1)=a_2(q^{(r_m-1)d}+q^{(r_m-2)d}+\ldots +q^{d}+1).
\end{equation}
From Corollary~\ref{cor:4} and the condition $\gcd(a_1,a_2)=1$ the required equalities follow.
\end{proof}

\begin{prop}\label{prop:6}
For a rational number $\l>0$ and an integer $q\ge2$, if the similarity transformation $H_\l$ keeps the space $X_q=\{q^n:n\in\Z\}$ in the same cloud, i.e., if $\l\in\L_{X_q}\cap\Q$, then
\begin{equation}\label{eq:10}
\l=q^\a\frac{1+q^d+\ldots+q^{(r_2-2)d}+q^{(r_2-1)d}}{1+q^d+\ldots+q^{(r_1-2)d}+q^{(r_1-1)d}}=
q^\a\frac{q^{r_2d}-1}{q^{r_1d}-1}
\end{equation}
for some integer $\a$, positive integer $d$, and coprime positive integers $r_1$ and $r_2$.
\end{prop}

\begin{proof}
By Corollary~\ref{cor:lambda-in-Q}, there exist $n>k$ and $m>l$ such that
$\l=\frac{q^n-q^k}{q^m-q^l}= q^{k-l}\frac{q^{n-k}-1}{q^{m-l}-1}$.
\end{proof}

\begin{prop}\label{prop:7}
If the square of the number $\l=\frac{q^n-1}{q^m-1}$ looks like the right hand side of Equation~$(\ref{eq:10})$, then $\l=1$.
\end{prop}

\begin{proof}
Without loss of generality, we can assume that $n>m$. Suppose that there exist positive integers $k$ and $l$ (with $k>l$) such that
$$
\l^2=\frac{(q^n-1)^2}{(q^m-1)^2}=\frac{q^k-1}{q^l-1}.
$$
Then
$$
(q^n-1)^2(q^l-1)=(q^m-1)^2(q^k-1),
$$
i.e.,
$$
q^{2n+l}-2q^{n+l}+q^{l}-q^{2n}+2q^{n}-1=q^{2m+k}-2q^{m+k}+q^{k}-q^{2m}+2q^{m}-1.
$$
From divisibility by the number $q$, it follows that $m=l$. After reduction, the following condition is obtained:
$$
q^{2n}-2q^n+1-q^{2n-m}+2q^{n-m}=q^{m+k}-2q^k+q^{k-l}-q^m+2
$$
Under the assumptions $n>m$ and $k>l$, it turns out that $1$ must be divisible by $q$.
\end{proof}

\begin{proof}[Proof of Theorem~$\ref{thm:2}$]
By Proposition~\ref{prop:6}, all rational numbers in the stabilizer are of the form~(\ref{eq:10}). Since the stabilizer is a multiplicative group, the square of any number from the stabilizer has the form~(\ref{eq:10}). It follows from Proposition~\ref{prop:7} that $\l$ has the required form.
\end{proof}

\begin{thm}
If a function $\v\:\N\to\R$ satisfies $\v(n+1)-\v(n)\to\infty$ as $n\to\infty$, then for any real numbers $q>1$ and $\l>0$, the similarity transformation $H_\l$ keeps the space $X_\v=\{q^{\v (n)}:n\in\N\}$ in the same cloud if and only if $\l=1$.
\end{thm}

\begin{proof}
Let $\l\ne1$. Suppose that the number $q$ is an integer, and the number $\l$ is rational. By Corollary~\ref{cor:lambda-in-Q}, there exist $n>k>M$ and $m>l>M$ such that $\l=\frac{q^{\v (n)}-q^{\v (k)}}{q^{\v (m)}-q^{\v (l)}}$.

If $n=m$, then
$$
\l=\frac{1-q^{\v(k)-\v(n)}}{1-q^{\v(l)-\v(m)}}\stackrel{M\to\infty}{\longrightarrow}1.
$$
If $M$ is sufficiently large, we get a contradiction with $\l\ne 1$.

If $n<m$, then
$$
\l=\frac{1-q^{\v(k)-\v(n)}}{q^{\v(m)-\v(n)}-q^{\v(l)-\v(n)}}\stackrel{M\to\infty}{\longrightarrow}0.
$$
If $M$ is sufficiently large, we get a contradiction with $\l>0$.

If $n>m$, then
$$
\l=\frac{q^{\v(n)-\v(m)}-q^{\v(k)-\v(m)}}{1-q^{\v(l)-\v(m)}}\stackrel{M\to\infty}{\longrightarrow}\infty.
$$
If $M$ is sufficiently large, we get a contradiction with $\l<\infty$.

If $q>1$ and $\l>0$ are arbitrary, one needs to return to the proof of Corollary~\ref{cor:lambda-in-Q}. Let us show that for $n>k>M$, $m>l>M$, and $\l\ne1$, the distortion tend to infinity. Here again there are three cases to consider. If $n=m$, then
$$
\dl=q^{\v(n)}-q^{\v(k)}-\l(q^{\v(m)}-q^{\v(l)})=
(1-\l)q^{\v(n)}\bigl(1-\frac{1}{1-\l}q^{\v(k)-\v(n)}+\frac{\l }{1-\l}q^{\v(l)-\v(n)}\bigr)\to\pm\infty.
$$
The cases $n<m$ and $n>m$ can be considered in the same way.
\end{proof}

\begin{cor}\label{cor:kvadrat-n}
If $\v(n)=n^2$, then for any real numbers $q>1$ and $\l>0$, the similarity transformation $H_\l$ keeps the space $X_\v=\{q^{n^2}:n\in\N\}$ in the same cloud if and only if $\l=1$.
\end{cor}


\begin{thebibliography}{15}
\bibitem{Mendelson} Mendelson E. \emph{An Introduction to Mathematical Logic\/} (4th ed.), London: Chapman and Hall/CRC, 1997.
\bibitem{Edwards} Edwards D. \emph{The Structure of Superspace. In: Studies in Topology}, ed. by Stavrakas
N.M. and Allen K.R., New York, London, San Francisco, Academic Press, Inc., 1975.
\bibitem{Gromov1981} Gromov M. \emph{Structures m\'etriques pour les vari\'et\'es riemanniennes}, edited by Lafontaine and Pierre Pansu, 1981.
\bibitem{Gromov1999} Gromov M. \emph{Metric structures for Riemannian and non-Riemannian spaces}, Birkh\"auser (1999). ISBN 0-8176-3898-9 (translation with additional content).
\bibitem{BurBurIva} Burago D., Burago Yu., Ivanov S. \emph{A Course in Metric Geometry}, AMS GSM 33, 2001.
\bibitem{Herron} Herron D.A. \emph{Gromov–Hausdorff Distance for Pointed Metric Spaces}, J. Anal., 2016, v. 24, N 1, pp 1–38.
\bibitem{Borzov} Borzov S.I., Ivanov A.O., Tuzhilin A.A. \emph{Extendability of Metric Segments in Gromov-Hausdorff Distance}. 2020, ArXiv e-prints, arXiv:2009.00458.
\bibitem{Borisova} \url{http://dfgm.math.msu.su/files/0students/2021-dip-Borisova.pdf}
\bibitem{IvaNikolaevaTuz} Ivanov A.O., Nikolaeva N.K., Tuzhilin A.A. \emph{The Gromov-Hausdorff Metric on the Space of Compact Metric Spaces is Strictly Intrinsic}. ArXiv e-prints, arXiv:1504.03830, 2015.
\bibitem{Memoli} Chowdhury S., Memoli F. \emph{Constructing Geodesics on the Space of Compact Metric Spaces}. ArXiv e-prints, arXiv:1603.02385, 2016.
\bibitem{IvaIliadisTuz} Ivanov A.O., Iliadis S., Tuzhilin A.A. \emph{Realizations of Gromov-Hausdorff Distance}. ArXiv e-prints, arXiv:1603.08850, 2016.
\end{thebibliography}
\end{document}